\documentclass[11pt,UTF8]{article}
\usepackage{bbm}
\usepackage[all,cmtip]{xy}
\usepackage{amsfonts,amssymb,amsmath,amscd,amsthm,mathtools}
\usepackage{mathrsfs}
\usepackage{latexsym}
\usepackage{color}
\usepackage{verbatim,enumitem}
\usepackage{hyperref}
\usepackage{tikz-cd}
\usepackage[OT1]{fontenc}

\input diagxy

\usepackage[total={150mm,230mm}]{geometry}

\DeclareMathOperator{\pt}{pt}
\DeclareMathOperator{\with}{\&}

\DeclareMathOperator{\thda}{{\rotatebox[origin=c]{-90}{$\twoheadrightarrow$}}}

\theoremstyle{plain}
  \newtheorem{thm}{Theorem}[section]
  \newtheorem{lem}[thm]{Lemma}
  \newtheorem{prop}[thm]{Proposition}
  
\theoremstyle{definition}
  \newtheorem{defn}[thm]{Definition}
  
  \newtheorem{exmp}[thm]{Example}
  \newtheorem{rem}[thm]{Remark}

\newtheorem*{SA}{Standing Assumption}

\newcommand{\ra}{\rightarrow}
\newcommand{\lra}{\longrightarrow}
\newcommand{\lam}{\lambda}

\newcommand{\QSup}{\sQ\text{-}{\sf Sup}}

\newcommand{\QOrd}{\sQ\text{-}{\sf Ord}}
\newcommand{\QFrm}{\sQ\text{-}{\sf Mod}_{\sf frm}}
\newcommand{\QTop}{\sQ\text{-}{\sf Top}}

\newcommand{\sub}{{\rm sub}}
\newcommand{\id}{{\rm id}}

\newcommand{\CP}{\mathcal{P}}
\newcommand{\CF}{\mathcal{F}}
\newcommand{\CO}{\mathcal{O}}

\newcommand{\sfm}{{\sf m}}

\newcommand{\sy}{{\sf y}}

\newcommand{\sQ}{{\sf Q}}
\newcommand{\bv}{\bigvee}
\newcommand{\bw}{\bigwedge}

\def\rto{{\bfig\morphism<180,0>[\mkern-4mu`\mkern-4mu;]\place(78,0)[\mapstochar]\efig}}

\begin{document}

\title{Sober  topological spaces valued in a   quantale}
\author{ Dexue Zhang, Gao Zhang \\    {\small   dxzhang@scu.edu.cn, gaozhang0810@hotmail.com}  }
\date{}
\maketitle
\begin{abstract}
The notion of sobriety is extended to the realm of topological spaces valued in a commutative and unital quantale, via an adjunction between  a category of quantale modules and the category of  quantale-valued topological spaces.   Relations between such sober spaces and quantale-valued domains based on flat ideals are investigated.

\vskip 2pt

\noindent\textbf{Keywords} Fuzzy topology; $\sQ$-order; $\sQ$-module; Flat ideal; Sober $\sQ$-topological space

\noindent \textbf{MSC(2020)} 54A40 54B30 18B30
\end{abstract}

\section{Introduction}
Sobriety is an interesting topological property for non-Hausdorff spaces and plays an important role in domain theory \cite{Gierz2003,Goubault}. A main feature of such a space  is that the space can be recovered from its open set lattice. Sobriety is often described via the   adjunction \[\mathcal{O}\dashv\pt \colon{\sf Frm}^{\rm op}\lra{\sf Top}\] between the opposite of the category  of frames  and the category  of topological spaces  \cite{Johnstone}. A topological space $X$ is sober  if the unit $\eta_X\colon  X\lra \pt \mathcal{O}(X)$ is a bijection (hence a homeomorphism).

Different extensions of the notion of sobriety to the fuzzy context have been proposed since 1990. Here we mention a few of the works on this topic:  Rodabaugh \cite{Rodabaugh}, Zhang and Liu \cite{ZL95}, Kotz\'{e} \cite{Kotze97}, Srivastava and Khastgir \cite{SK98}, Pultr and Rodabaugh \cite{PR08a,PR08b},   Guti\'{e}rrez Garc\'{i}a,   H\"{o}hle and de Prada Vicente \cite{GHP},  Yao \cite{Yao11,Yao12}, J\"{a}ger and  Yao \cite{JY2016},  Singh and Srivastava \cite{SS16}, and etc. In these works, the table of truth-values is assumed to be a complete Heyting algebra (i.e., a frame),  sometimes even a completely distributive lattice. For such a table of truth-values, the logic connective \emph{conjunction},   modeled by the meet operation of the lattice, is idempotent. From the viewpoint of fuzzy logic, this is a serious drawback. For example,   BL-algebras \cite{Ha98}  and left-continuous t-norms \cite{Klement2000} are seldom idempotent.

The aim of this paper is to extend the notion of sobriety to the realm of topological spaces valued in a commutative and unital quantale. Complete Heyting algebras, BL-algebras, and the interval $[0,1]$ equipped with a left-continuous t-norm   are typical examples of such quantales.

This paper is a companion of \cite{Zhang2018}, which focuses on sobriety of quantale-valued \emph{cotopological} spaces.
In the classical context, a topological space can be described by open sets as well as by  closed sets, and we can switch between open sets and closed sets by taking complements. Thus,  it makes no difference whether we choose to work with closed sets or with open sets in the classical context. But, in the fuzzy context, the table of truth-values usually does not satisfy the law of double negation,  there is no natural way to switch between open sets and closed sets, so, as observed in \cite{Zhang2018}, in fuzzy topology it is often necessary to consider the topology version (in term of open sets) and the cotopology version (in term of closed sets) of the same concept. In other words, like  a coin, fuzzy topology was born with two sides. And, we would like to stress that, due to the lack of symmetry in the table of truth-values,   different ideas and techniques  are needed in the topology and the cotopology versions of the same theory.

In this paper, like that in \cite{Zhang2018}, we emphasize the connection between quantale-valued orders and quantale-valued topologies. For each commutative and unital quantale $\sQ$, a category of $\sQ$-modules (or equivalently, complete $\sQ$-lattices) is specified, sober $\sQ$-topological spaces are then described via an adjunction between this category and that of $\sQ$-topological spaces.  It is shown that for a commutative and integral quantale, the Scott $\sQ$-topology of every $\CF$-domain (see Definition \ref{defn of F-domain}) is sober. In the final section, some  examples are presented in the case that the quantale is the interval $[0,1]$ together with a continuous t-norm.

\section{Quantales and quantale-valued orders}
In this section we recall some basic ideas about quantale-valued orders and fix some notations.

A \emph{frame} \cite{Johnstone} is a complete lattice $L$ such that   the binary meet operation $\wedge$ distributes over arbitrary joins; that is, $x\wedge\bv S=\bv\{x\wedge s\mid s\in S\}$ for each $x\in L$ and each $S\subseteq L$. A \emph{frame map} $f\colon L_1\lra L_2$ between frames is a map that preserves arbitrary joins and finite meets, including the empty ones. Frames and frame maps constitute a category \[\sf Frm.\]

Following Rodabaugh \cite[page 303]{Rodabaugh}, we say that a map $f\colon L\lra M$ between complete lattices (not necessarily frames) is  \emph{frame-like} if $f$ preserves arbitrary joins and finite meets. 

A  \emph{quantale} \cite{Rosenthal1990}
\(\sQ=(\sQ,\with)\)
is a semigroup such that the underlying set $\sQ$ is a complete lattice   and that the multiplication $\with$ distributes over arbitrary joins on both sides. As usual, we write $0$ and $1$ for the bottom and the top element of $\sQ$, respectively.  A quantale $\sQ$ is \begin{itemize} \setlength{\itemsep}{0pt} \item  commutative  if the multiplication $\with$ is commutative; \item unital if it has a unit element $k$; \item integral  if it is unital and the unit $k$ is the top element  of $\sQ$.\end{itemize}

It is clear that a complete lattice $L$ is a frame if, and only if, $(L,\wedge)$ is a quantale. In this paper, when we say that a quantale $\sQ=(\sQ,\with)$ is a frame, we mean the semigroup operation $\with$ of $\sQ$ is the meet operation $\wedge$. It is possible that the underlying lattice of a   quantale   is a frame, but the quantale itself is not a frame.

\begin{SA}In this paper, by a quantale we always mean a commutative and unital one, with unit denoted by $k$, unless otherwise specified.\end{SA}

Given a quantale $\sQ$, the  multiplication $\&$ determines a binary operator $\ra$,  known as the \emph{implication operator} of $\&$, via the adjoint property:
\[p\with q\leq r\iff q\leq p\ra r.\]

A \emph{$\sQ$-valued order}, a $\sQ$-order for short, is a map $\alpha\colon X\times X\lra \sQ$ such that for all $x,y,z\in X$,
$$\alpha(x,x)\geq k\quad \text{and}\quad  \alpha(y,z)\with \alpha(x,y)\leq\alpha(x,z).$$
The pair $(X,\alpha)$  is call a \emph{$\sQ$-ordered set}, or  a \emph{$\sQ$-category} in the language of enriched categories \cite{Lawvere1973}.  It is customary to write $X$ for the pair $(X,\alpha)$ and write $X(x,y)$ for $\alpha(x,y)$.

A map $f\colon X\lra Y$ between $\sQ$-ordered sets is said to \emph{preserve $\sQ$-order}   if $X(x,y)\leq Y(f(x),f(y))$ for all $x,y\in X$.
$\sQ$-ordered sets and $\sQ$-order-preserving maps constitute a category
 \[\sQ\text{-}\sf Ord.\]

If $\alpha$ is a $\sQ$-order on $X$, it follows from the commutativity of $\with$ that $\alpha^{\rm op}(x,y)\coloneqq\alpha(y,x)$ is also a $\sQ$-order on $X$, called the opposite of $\alpha$.

The following two examples belong to the folklore in the theory of $\sQ$-orders, it is hard to specify where they appeared for the first time.
\begin{exmp}
For all $x,y\in \sQ$, let
$$\alpha_L(x,y)=x\ra y.$$ Then $\alpha_L$ is a   $\sQ$-order, called the \emph{canonical $\sQ$-order}, on $\sQ$. The opposite of $\alpha_L$ is denoted by $\alpha_R$; that is, $$  \alpha_R(x,y)= y\ra x.$$
\end{exmp}

\begin{exmp}\label{inclusion Q-order} For each set $X$, the map \[\sub_X\colon \sQ^X\times\sQ^X\lra\sQ, \quad  \sub_X(\phi,\psi)=\bigwedge_{x\in X}\phi(x)\ra\psi(x)\] is a  $\sQ$-order on the set $\sQ^X$, known as the  \emph{inclusion $\sQ$-order}.   If $X$ is a singleton set, then $(X,\sub_X)$ degenerates to the $\sQ$-ordered set $(\sQ,\alpha_L)$.
\end{exmp}

The \emph{underlying order} of a $\sQ$-ordered set $X$ refers to the order $\leq$ given by $$x\leq   y \quad\quad  {\rm if}~ k\leq X(x,y) .$$ For convenience, for each $\sQ$-ordered set $X$, we shall write $X_0$ for   $X$ with the underlying order.

Two elements $x,y$ of a $\sQ$-ordered set $X$ are \emph{isomorphic} if \(k\leq X(x,y)\wedge X(y,x).\) A $\sQ$-ordered set $X$ is \emph{separated} if isomorphic elements are identical; that is, \[k\leq X(x,y)\wedge X(y,x) \implies x=y.\]

Let $X,Y$ be $\sQ$-ordered sets; let $f\colon X\lra Y$ and $g\colon Y\lra X$ be  maps.  We say that $f$ is \emph{left adjoint} to $g$ (and/or, $g$ is \emph{right adjoint} to $f$) and write $f\dashv g$, if
$$Y(f(x),y)=X(x,g(y))$$
for all $x\in X$ and $y\in Y$.
This is a special case of enriched adjunctions in  category theory \cite{Borceux1994,Lawvere1973}.
\begin{thm}\label{Characterization of adjoints} {\rm(\cite[page 295]{Stubbe2006})} Let $f\colon X\lra Y$ and $g\colon Y\lra X$ be a pair of maps between $\sQ$-ordered sets. Then $f$ is left adjoint to $g$ if, and only if, the following conditions are satisfied: \begin{enumerate}[label={\rm(\roman*)}] \setlength{\itemsep}{0pt}
  \item Both $f$ and $g$ preserve $\sQ$-order. \item The map $f\colon X_0\lra Y_0$ is left adjoint to $g\colon Y_0\lra X_0$; that is,  for all $x\in X$ and $y\in Y$, $f(x)\leq y\iff x\leq g(y)$. \end{enumerate}\end{thm}

Let $X$ be a $\sQ$-ordered set. A  \emph{weight}   of  $X$ is a map $\phi\colon  X\lra \sQ$ such that for all $x,y\in X$,
$$\phi(y)\with X(x,y)\leq\phi(x).$$

\begin{rem}There exist different terminologies for a weight  of a $\sQ$-ordered set $X$. Firstly, since the above inequality can be read as ``that $y$ belongs to $\phi$ and $x$ is smaller than or equal to $y$ implies $y$ belongs to $\phi$'', a weight is also called a \emph{lower fuzzy set} of $X$.
Secondly, since a weight  of $X$ is just a $\sQ$-order-preserving map (or, a $\sQ$-functor) $\phi\colon X^{\rm op}\lra(\sQ,\alpha_L)$, it is also called a $\sQ$-presheaf of $X$ (viewed as an enriched category). The terminology \emph{weight} adopted here comes from category theory, see e.g. \cite{KS2005}. \end{rem}

The weights of a $\sQ$-ordered set $X$ constitute a $\sQ$-ordered set $\mathcal{P}X$ with
$$\mathcal{P}X(\phi_1,\phi_2)\coloneqq\sub_X(\phi_1,\phi_2).$$

For each $x\in X$, $X(-,x)$ is a weight of $X$  and we have the following:

\begin{lem}[Yoneda lemma]
Let $X$ be a $\sQ$-ordered set and $\phi$ be a weight of $X$, then
$$\mathcal{P}X(X(-,x),\phi)=\phi(x).$$
\end{lem}

The Yoneda lemma entails that the map
$$\sy_X\colon X\lra\mathcal{P}X, \quad x\mapsto X(-,x)$$  is   an embedding if $X$ is separated. By abuse of language,  we call it the \emph{Yoneda embedding} no matter $X$ is separated or not.

Each $\sQ$-order-preserving map $f\colon X\lra Y$ gives rise to a natural adjunction between $\CP X$ and $\CP Y$. Precisely, the map \[f^\ra\colon \mathcal{P}X\lra\mathcal{P}Y,\quad f^\ra(\phi)(y)=\bigvee_{x\in X} \phi(x)\with Y(y,f(x)) \]
is left adjoint to \[f^\leftarrow\colon \mathcal{P}Y\lra\mathcal{P}X, \quad f^\leftarrow(\psi)(x)=\psi(f(x)).\]

Let $X$ be a $\sQ$-ordered set and let  $\phi$ be a weight of $X$. We say that an element $a$ of $X$ is a  \emph{supremum}  of $\phi$ and write $a=\sup\phi$, if
$$X(a,y)=\CP X(\phi, X(-,y))$$
for all $y\in X$. In the language of category theory, $\sup\phi$ is the colimit of the identity functor $\id\colon X\lra X$ weighted by $\phi$ (see e.g. \cite{Borceux1994,Stubbe2005}). Suprema of a weight, if exist, are unique up to isomorphism.

We say that a $\sQ$-ordered set $X$ is \emph{cocomplete} if every weight of $X$ has a supremum. It is clear that $X$ is cocomplete if, and only if, the Yoneda embedding $\sy_X\colon X\lra\mathcal{P}X$ has a left adjoint (see e.g. \cite{Stubbe2005}).

\begin{exmp}\label{sup in Q} The $\sQ$-ordered set $(\sQ,\alpha_L)$ is cocomplete. For each weight $\phi$ of $(\sQ,\alpha_L)$,  \[\sup\phi=\bv_{x\in\sQ}\phi(x)\with x = \phi(k).\]
\end{exmp}

A  \emph{coweight}  of a $\sQ$-ordered set $X$ is defined to be weight of the opposite $X^{\rm op}$ of $X$. Explicitly,  a coweight  of $X$ is a $\sQ$-order-preserving map $\psi\colon X\lra(\sQ,\alpha_L)$.  Coweights are also known as  \emph{upper  fuzzy sets} and  \emph{covariant $\sQ$-presheaves}.
The coweights of $X$ constitute a $\sQ$-ordered set $\mathcal{P}^\dagger X$ with
$$\mathcal{P}^\dagger X(\psi_1,\psi_2)\coloneqq\sub_X(\psi_2,\psi_1).$$

Let $\psi$ be a coweight and let $a$ be an element  of $X$. We say that $a$ is an  \emph{infimum}  of $\psi$  if
$$X(y,a)=\CP^\dagger X(X(y,-),\psi)$$
for all $y\in X$.

A $\sQ$-ordered set  $X$ is \emph{complete} if every coweight of $X$ has an infimum.
It is known  \cite{Stubbe2005} that
\begin{enumerate}[label={\rm(\roman*)}] \setlength{\itemsep}{0pt}
  \item
 a $\sQ$-ordered set $X$ is cocomplete if, and only if, $X$ is complete;  and
\item a $\sQ$-order-preserving map $f\colon X\lra Y$ between cocomplete $\sQ$-ordered sets is a left adjoint if, and only if, $f$ \emph{preserves suprema} in the sense that \(f(\sup\phi)=\sup f^\ra(\phi)\) for each weight $\phi$ of $X$.
\end{enumerate}

A $\sQ$-ordered set is  a \emph{complete $\sQ$-lattice} if $X$ is separated and  complete (or equivalently, cocomplete). Complete $\sQ$-lattices and left adjoints constitute a category   \[\QSup.\]

For each $\sQ$-ordered set $X$, $\CP X$ is cocomplete, hence a complete $\sQ$-lattice \cite{Stubbe2005}. Actually, \[\sy_X^\leftarrow\colon \CP\CP X\lra \CP X \] is a left adjoint of  $\sy_{\CP X}\colon \CP X\lra\CP\CP X$.
This means that the supremum of a weight $\Phi$ of $\CP X$ is given by \[\sup\Phi=\sy_X^\leftarrow(\Phi)=\bv_{\phi\in\CP X} \Phi(\phi)\with\phi.\]

The correspondence \[f\colon X\lra Y\quad \mapsto\quad f^\ra\colon \CP X\lra\CP Y\] defines a functor $\CP\colon \QOrd\lra\QSup$ that is left adjoint to the forgetful functor $\QSup\lra\QOrd$ \cite{Stubbe2005}. We write  \[\mathbb{P}=(\CP,\sfm,\sy)\] for the monad on $\QOrd$ arising from this adjunction.    Explicitly,   \begin{itemize} \setlength{\itemsep}{0pt}\item  for each $\sQ$-order-preserving map $f\colon X\lra Y$, $\CP f=f^\ra\colon \CP X\lra\CP Y$; \item the unit is the Yoneda embedding $\sy_X\colon X\lra\CP X$; \item the multiplication  \(\sfm_X={\sup}_{\CP X}=\sy_X^\leftarrow\colon \CP\CP X\lra \CP X\). \end{itemize}

The category of $\mathbb{P}$-algebras and $\mathbb{P}$-homomorphisms is precisely that of   $\sQ$-complete lattices and left adjoints (see e.g. \cite{LZ2020,Stubbe2005}); that is, \[\mathbb{P}\text{-}{\sf Alg}=\QSup.\]

As we see below, there is a very useful and convenient characterization of $\QSup$, it is equivalent to the category of $\sQ$-modules.   Recall that the category {\sf Sup} of complete lattices and join-preserving maps is symmetric and monoidal closed, see  \cite[Chapter I]{Joyal-Tierney}, thus, as in any such categories \cite{MacLane1998}, one can talk about monoids and monoid actions in {\sf Sup}. It is readily seen that a commutative and unital quantale $\sQ$ is exactly a commutative monoid in {\sf Sup}. So, there are $\sQ$-actions and $\sQ$-modules.

\begin{defn} (\cite{Joyal-Tierney}) A   $\sQ$-module (precisely, a left $\sQ$-module) is a   pair $(X,\otimes)$, where $X$ is a complete lattice and $\otimes\colon \sQ\times X \lra X$ is a   map, called a (left) $\sQ$-action on $X$,  subject to the following conditions:   for all $x\in X$ and $r,s, \in \sQ$, \begin{enumerate}[label={\rm(\roman*)}] \setlength{\itemsep}{0pt}
  \item $k\otimes  x=x$, where $k$ is the unit of $\sQ$; \item $ s\otimes (r\otimes  x)  =(s\with r)\otimes x$; \item   $r\otimes -\colon X\lra X$ preserve joins;     and \item $-\otimes x\colon \sQ\lra X$ preserve joins.
\end{enumerate}\end{defn}
The conditions (iii) and (iv) together amount to requiring that $\otimes:\sQ\times X\lra X$ is a bimorphism in {\sf Sup} (see e.g. \cite[page 5]{Joyal-Tierney}). Thus, by universal property of tensor products  in {\sf Sup}, one sees that the above definition of  $\sQ$-modules, which is a bit more reader-friendly, is equivalent to that in \cite{Joyal-Tierney}.

A   homomorphism  $f\colon (X,\otimes)\lra(Y,\otimes)$ between $\sQ$-modules is a join-preserving map $f\colon X\lra Y$   that preserves the action, i.e., $ r\otimes f(x)= f(r\otimes x)$ for all $r\in \sQ$ and $x\in X$.
$\sQ$-modules and   homomorphisms constitute a category \[\sQ\text{-}{\sf Mod}.\]

If $\sQ$ is the Boolean algebra $(\{0,1\},\wedge)$, then $\sQ\text{-}{\sf Mod}$ is   equivalent to the category ${\sf Sup}$ of complete lattices and join-preserving maps    \cite[page 8, Proposition 1]{Joyal-Tierney}.

The quantale $(\sQ,\with)$ itself is a $\sQ$-module, with $\with$ being   the action. We often view $\sQ$ as a $\sQ$-module in this way.

Let $X$ be a $\sQ$-ordered set. For each $x\in X$ and  $r\in \sQ$, the \emph{tensor of $r$ and $x$} (see e.g. \cite[page 288]{Stubbe2006}), denoted by $r\otimes x$, is a point of $X$ such that for all $y\in X$, \[X(r\otimes x,y)= r\ra X(x,y).\]

A $\sQ$-ordered set $X$ is   \emph{tensored} if the tensor $r\otimes x$ exists for all $x\in X$ and $r\in \sQ$. Some  facts about tensors are listed below, which can be found in  \cite{Stubbe2006} and \cite{LZ2009}. \begin{itemize} \setlength{\itemsep}{0pt} \item  Every  complete $\sQ$-lattice is tensored. Actually, $r\otimes x$ is the supremum of the weight $r\with X(-,x)$. \item  Every left adjoint $f$ preserves tensors; that is, $r\otimes f(x)= f(r\otimes x)$. \item A map $f\colon X\lra Y$ between complete $\sQ$-lattices preserves $\sQ$-order  if, and only if, $f\colon X_0\lra Y_0$ preserves order and $r\otimes  f(x) = f( r\otimes x)$ for all $x\in X$ and $r\in \sQ$.  \end{itemize}

Given a complete $\sQ$-lattice $X$, taking tensor \((r,x)\mapsto r\otimes x\) makes $X$ into a $\sQ$-module. Conversely, given a $\sQ$-module $(X,\otimes)$, if we let \[\alpha(x,y)=\bv\{r\in\sQ\mid r\otimes x\leq y\},\] then $(X,\alpha)$ is a complete $\sQ$-lattice. These two processes are inverse to each other. So we have the following characterization of $\QSup$, which is   an instance of  \cite[Corollary 4.13]{Stubbe2006}.
\begin{prop} The category $\sQ\text{-}{\sf Mod}$ is equivalent to  the category $\QSup$. \end{prop}

\begin{exmp}For each set $X$, the assignment $(r,\lambda)\mapsto r\with\lambda$ defines a $\sQ$-action on $\sQ^X$, so, $\sQ^X$ becomes a $\sQ$-module. We shall always view $\sQ^X$ as a $\sQ$-module in this way. It is clear that the corresponding $\sQ$-order of the $\sQ$-module $\sQ^X$ is just the inclusion $\sQ$-order $\sub_X$ in Example \ref{inclusion Q-order}. \end{exmp}

For more information on $\sQ$-modules and their relations to $\sQ$-orders the reader is  referred to the monograph \cite{EGHK}.

\section{Sober $\sQ$-topological spaces}

A \emph{$\sQ$-topology} on a set $ X $ is subset $ \tau $ of $\sQ^X $ subject to the following conditions:
\begin{enumerate}[label={\rm(O\arabic*)}] \setlength{\itemsep}{0pt}
\item The constant map $1_X\colon X\lra\sQ$ with value $1$ belongs to $\tau$;
\item  $ \lambda\wedge\mu\in\tau $ for all $ \lambda,\mu\in\tau $;
\item    $ \bigvee_{j\in J}\lambda_{j}\in \tau $ for each subset $\{\lambda_{j}\}_{j\in J}$ of $ \tau $; \item $\lam\with r \in\tau$ for all $r\in \sQ$ and $\lam\in \tau$, where  $(\lam\with r)(x)=\lam(x)\with r$ for all $x\in X$.
\end{enumerate} The pair $(X,\tau)$ is called a  $\sQ$-topological space; elements of $\tau$ are called open sets.



It is customary to write $X$ for a $\sQ$-topological space and write $\mathcal{O}(X)$ for the set of open sets of $X$.
The interior operator of a $\sQ$-topological space $X$ refers to the map
\[(-)^\circ\colon \sQ^X\lra\sQ^X,\quad
\lambda^\circ=\bv\{\mu\in\mathcal{O}(X)\mid\mu\leq\lambda\}.\]  The interior operator  satisfies the following conditions: for all $\lambda,\mu\in\sQ^X$, (int1) $1_X^\circ=1_X$;
(int2) $ \lambda^\circ \leq \lambda$;
(int3) $(\lambda\wedge\mu)^\circ= \lambda^\circ\wedge\mu^\circ$;
(int4) \(\sub_X(\lambda,\mu)\leq\sub_X(\lambda^\circ,\mu^\circ)\); and
(int5) $(\lambda^\circ)^\circ=\lambda^\circ$.
Actually, $\sQ$-topologies on $X$ correspond bijectively to operators $\sQ^X\lra\sQ^X$ satisfying (int1)-(int5). This fact also justifies the definition of $\sQ$-topology.



A map $f\colon X\lra Y$ between  $\sQ$-topological spaces is continuous if $\lam\circ f$ is an open set of $X$ for each open set $\lam$ of $Y$.  The category of $\sQ$-topological spaces and continuous maps is denoted by \[\QTop.\]

Given a $\sQ$-topological space $X$, the $\sQ$-relation \[\Omega X\colon X\times X\lra\sQ,\quad \Omega X(x,y)=\bw_{\lam\in\CO(X)}\lam(x)\ra\lam(y)\] is a $\sQ$-order, called the \emph{specialization $\sQ$-order} of the space $X$. In this way, we obtain a functor \[\Omega\colon \QTop\lra\QOrd.\] It is known that $\Omega$ has a left adjoint (see e.g.  \cite{LZ2006}).

We say that a $\sQ$-topological space $X$ is   $T_0$  if for any pair $x,y$ of distinct points, there is an open set $\lam$ such that $\lam(x)\not=\lam(y)$. It is clear that $X$ is $T_0$ if, and only if, the specialization $\sQ$-order of $X$ is separated.

Write \[\QFrm\] for the category composed of   $\sQ$-modules and  frame-like $\sQ$-module homomorphisms, that is,   $\sQ$-module homomorphisms that preserves finite meets.
In the case that $\sQ$ is the two-element Boolean algebra, $\QFrm$ coincides with the category {\sf CSLF}   of complete lattices and frame-like maps in Rodabaugh \cite{Rodabaugh}.

For each $\sQ$-topological space $X$, by commutativity of $\&$, the set $\mathcal{O}(X)$  of open sets of $X$ is a   $\sQ$-module with $\sQ$-action \[r\otimes \lambda\coloneqq  \lambda\with r.\]
The correspondence $X\mapsto\CO(X)$ gives rise to a functor \[\mathcal{O}\colon \sQ\text{-}{\sf Top}\lra \QFrm^{\rm op}.\]

Given a $\sQ$-module $(A,\otimes)$, by a \emph{point} of $(A,\otimes)$ we mean a frame-like $\sQ$-module homomorphism $$p\colon (A,\otimes)\lra (\sQ,\with).$$ Explicitly, a point of $(A,\otimes)$ is a map $p\colon A\lra\sQ$  subject to the following conditions:  \begin{enumerate}[label={\rm (pt\arabic*)}]\item $p(\top)=1$, where $\top$ is the top element of $A$; \item $p(\lambda\wedge \mu)=p(\lambda)\wedge p(\mu)$; \item $p\big(\bv_{i\in I}\lambda_i\big) = \bv_{i\in I}p(\lambda_i)$; \item $p(r\otimes\lambda) =  r\with p(\lambda)$. \end{enumerate}

\begin{defn}\cite{Zhang2018} \label{defn of sober} A $\sQ$-topological space $X$ is sober if  for each point $p\colon \mathcal{O}(X)\lra\sQ$ of the $\sQ$-module $\CO(X)$,
there is a unique point $x$ of $X$ such that $p(\lambda)=\lambda(x)$ for all   $\lambda\in\mathcal{O}(X)$.\end{defn}

\begin{rem}[When $\sQ$ is a frame, I] \label{frame-valued II}
Assume that $\sQ=(\sQ,\with)$ is a frame, i.e., $\with=\wedge$. For each $\sQ$-topological space $X$, the assignment $r\mapsto r_X$ defines a frame map $i_X\colon\sQ\lra \CO(X)$, where $r_X$ denotes the constant open set with value $r$. It is clear that a map $p\colon \CO(X)\lra\sQ$ is a point of the $\sQ$-module $\CO(X)$  if, and only if, $p$ is a frame map such that $p\circ i_X$ is the identity map. Thus, for such a quantale, a $\sQ$-topological space is sober in the sense of Definition \ref{defn of sober} if, and only if, it is sober in the sense of \cite{ZL95}. \end{rem}

\begin{exmp} (See \cite[Example 3.2(b)]{SK98} in the case that $(\sQ,\with)$ is a frame) Write  $\mathbb{S}$ for the  $\sQ$-topological space obtained by endowing $\sQ$ with the coarsest $\sQ$-topology that has the identity map as an open set. We  call $\mathbb{S}$ the \emph{Sierpi\'{n}ski $\sQ$-topological space}.
It is readily verified that the specialization $\sQ$-order of $\mathbb{S}$ is the canonical $\sQ$-order on $\sQ$; that is,  $\Omega\thinspace\mathbb{S}=(\sQ,\alpha_L)$.

If $(\sQ,\with)$ satisfies the following requirements: \begin{itemize} \setlength{\itemsep}{0pt}\item $(\sQ,\with)$ is integral, \item the underlying lattice $\sQ$ is a frame, \item $(a\wedge b)\with r= (a\with r)\wedge(b\with r)$ for all $r,a,b\in\sQ$,\end{itemize}
then, by help of the identity \((a\vee(x\with r))\wedge b = (a\wedge b)\vee(b\wedge(x\with r)) \) one sees that     \[\CO(\mathbb{S})= \{(a\vee(\id\with r))\wedge b\mid r,a,b\in\sQ\}.\]
Thus, each map $p\colon\CO(\mathbb{S})\lra\sQ$ satisfying (pt1)-(pt4) is determined by its value at the identity map  on $\sQ$. From this one derives that $p(\lambda)=\lambda(a)$ for all $\lambda\in\CO(\mathbb{S})$, where $a=p(\id)$, hence $\mathbb{S}$ is sober. \end{exmp}

As one expects, sobriety of $\sQ$-topological spaces can be described via an adjunction  between the categories $\QTop$ and $\QFrm^{\rm op}$.
Given a $\sQ$-module $(A,\otimes)$, write \[\pt A\] for the set of all points of $(A,\otimes)$. For each $\lambda\in A$, define \[\widehat{\lambda}\colon \pt A\lra\sQ\] by \[\widehat{\lambda}(p)=p(\lambda).\] Then \[\widehat{A}\coloneqq\{\widehat{\lambda}\mid \lambda\in A\}\] is a $\sQ$-topology on $\pt A$ and the assignment \[(A,\otimes)\mapsto (\pt A, \widehat{A})\] defines a functor \[\pt \colon \QFrm^{\rm op}\lra\QTop.\]

\begin{prop}  $\pt \colon \QFrm^{\rm op}\lra\QTop$ is  right adjoint to $\mathcal{O}\colon \sQ\text{-}{\sf Top}\lra \QFrm^{\rm op}$.\end{prop}

\begin{proof}We only spell out  the unit and the counit of the adjunction here.
The unit:  for each $\sQ$-topological space $X$, \[\eta_X\colon X\lra \pt \mathcal{O}(X),\quad \eta_X(x)(\lambda)=\lambda(x).\]
The counit:  for each $\sQ$-module $(A,\otimes)$, \begin{align*}\varepsilon_A\colon  A\lra \mathcal{O}(\pt A),& \quad \varepsilon_A(\lambda)=\widehat{\lambda}.\qedhere\end{align*} \end{proof}

\begin{prop}
Let $X$ be a $\sQ$-topological space and let $(A,\otimes)$ be a $\sQ$-module.   \begin{enumerate}[label={\rm(\roman*)}] \setlength{\itemsep}{0pt}
  \item $X$ is a $T_0$ space if, and only if, $\eta_X\colon X\lra \pt \mathcal{O}(X)$ is injective. \item  $X$ is sober if,  and only if,    $\eta_X\colon X\lra \pt \mathcal{O}(X)$ is bijective, hence a homeomorphism. \item  The   $\sQ$-topological space $\pt A$ is  sober.    In particular,  the space $\pt \mathcal{O}(X)$ is sober and  is called the sobrification of $X$. \end{enumerate} \end{prop}
\begin{proof}The verification is standard. For example, we check that if $\eta_X\colon X\lra \pt \mathcal{O}(X)$ is bijective, then it is a homeomorphism. For this, we check that $\eta_X$ is an open map in this case.
Write $\hat{x}$ for the point $\eta_X(x)$ of $\pt \mathcal{O}(X)$. Since $\eta_X$ is bijective, every point of $\pt \mathcal{O}(X)$ is of the form $\hat{x}$ for a unique $x\in X$.  Since for each open set $\lambda$  and each point $x$  of $X$,   \(\widehat{\lambda} (\hat{x})=\eta_X(x)(\lambda)= \lambda (x)\), it follows that the image of $\lambda$ under $\eta_X$ is  $\widehat{\lambda}$, hence  $\eta_X$ is an open map. \end{proof}
A $\sQ$-module $(A,\otimes)$ is   \emph{spatial} if the counit $\varepsilon_A\colon  A\lra \mathcal{O}(\pt A)$ is an isomorphism. It is clear that for each $\sQ$-topological space $X$, the $\sQ$-module $\CO(X)$ is spatial. The category of sober $\sQ$-topological spaces is dually equivalent to the subcategory of $\QFrm$ consisting spatial $\sQ$-modules, and is  a reflective full subcategory of $\QTop$.

\begin{rem}[When $\sQ$ is a frame, II] \label{frame-valued III}
Assume that $\sQ=(\sQ,\with)$ is a frame, i.e., $\with=\wedge$. Let $i\colon \sQ\lra A$ be an object of the slice category $\sQ\downarrow{\sf Frm}$; that is, $i$ is a frame map. Define $\otimes\colon \sQ\times
A\lra A$ by $r\otimes x=x\wedge i(r)$. Then $(A,\otimes)$ is   a $\sQ$-module. This assignment makes  $\sQ\downarrow{\sf Frm}$ into a full subcategory of    $\QFrm$   containing all  spatial $\sQ$-modules. There is a nice characterization of this subcategory: a $\sQ$-module $(A,\otimes)$  is generated by an object in the slice category  $\sQ\downarrow{\sf Frm}$ if, and only if, for each $x\in A$ the map $x\wedge -\colon A\lra A$ is a $\sQ$-module homomorphism.
This characterization is essentially Proposition 3.7 and Proposition 3.9 in \cite{Yao12}. For the sake of completeness, a direct verification is included here.

Necessity is obvious, we only need  to check the sufficiency. First, since $x\wedge -\colon A\lra A$ is a $\sQ$-module homomorphism for all $x\in A$, it follows that, in $A$,  binary meets distribute over arbitrary joins, hence $A$ is a frame.  Define $i\colon \sQ\lra A$ by $i(r)=r\otimes\top$, where $\top$ is the top element of the frame $A$. We claim that $i$ is a frame map and the $\sQ$-module $(A,\otimes)$ is generated by $i$. That $i$  preserves join is clear,  to see that it preserves finite meet, let $r,s\in \sQ$. Then \begin{align*}i(r)\wedge i(s)&= (r\otimes\top)\wedge(s\otimes\top)\\ &= s\otimes ((r\otimes\top)\wedge\top)  \\ &= s\otimes (r\otimes\top)\\ &=(s\wedge r) \otimes\top \\ &= i(r\wedge s),\end{align*}where the second equality holds because $(r\otimes \top)\wedge-\colon A\lra A$ is a $\sQ$-module homomorphism. It remains to check that $(A,\otimes)$ is generated by $i$. This is easy, since for all $r\in\sQ$ and $x\in A$,  \[r\otimes x=r\otimes(x\wedge \top)=x\wedge(r\otimes \top)= x\wedge i(r).\]
\end{rem}

\begin{rem}Let $\sQ=(\sQ,\with)$ be a commutative and integral quantale, viewed as a $\sQ$-module. It is trivial that $(\sQ,\with)$ is spatial. But, the map \[r\wedge-\colon (\sQ,\with)\lra (\sQ,\with)\] is a  $\sQ$-module homomorphism for all $r\in\sQ$ if, and only if,   the quantale $(\sQ,\with)$ is a frame, i.e., $\with=\wedge$. Sufficiency is clear, to see the necessity, let $r\in\sQ$. Since $r\wedge-$ preserves the $\sQ$-action, it follows that
\[r=r\wedge(r\with 1)=r\with(r\wedge1)  =r\with r,\] hence $r$ is idempotent and consequently, $(\sQ,\with)$ is a frame. \end{rem}

\section{Sobriety of $\CF$-domains}
It is well-known in domain  theory that the Scott topology of a domain (= continuous dcpo) is sober   \cite{Gierz2003,Goubault}. This section concerns an analogy of this conclusion in the enriched context when $\sQ$ is an integral and commutative quantale.


For each weight $ \phi $ and each coweight $ \psi $ of a $\sQ$-ordered set $ X $, let \[ \phi \pitchfork \psi=\bv_{x\in X}\phi(x)\with \psi(x). \]

It is easily verified that for each weight $ \phi $, each coweight $ \psi $, and each point $a$ of a $\sQ$-ordered set $ X $, \[\phi\pitchfork X(a,-)=\phi(a) \quad \text{and} \quad X(-,a)\pitchfork\psi=\psi(a).\]

\begin{rem}From the viewpoint of fuzzy set theory, the value $\phi \pitchfork \psi$ measures the degree that $\phi$ and $\psi$  have a common point when $\phi$  and $\psi$ are viewed as fuzzy subsets of $X$.   From the perspective of category theory, $\phi \pitchfork \psi$ is the colimit of  $\psi\colon X\lra (\sQ,\alpha_L)$ weighted by $\phi$, i.e., the supremum of $\psi^\ra(\phi)$ in $(\sQ,\alpha_L)$ (c.f. \cite[Example 2.11]{LZZ2020}).  Furthermore, if we view $\phi$ as a fuzzy relation $X\rto*$  from $X$ to  a singleton set and view $\psi$ as a fuzzy relation $*\rto X$ from a singleton set to $X$, then $\phi \pitchfork \psi$ is the composite fuzzy relation $\phi\circ\psi\colon *\rto X\rto*$. \end{rem}

\begin{defn}(\cite{LZZ2020,TLZ2014,SV2005})
A  flat ideal of a $\sQ$-ordered $ X $ is a weight  $\phi $ of $X$  that is (i)   inhabited in the sense that $\bv_{x\in X}\phi(x)\geq k$; and (ii)   flat in the sense that for any coweights $\psi_1, \psi_2$ of $X$, \[\phi \pitchfork (\psi_1\wedge\psi_2)= (\phi\pitchfork\psi_1)\wedge (\phi\pitchfork\psi_1).  \]
\end{defn}

Flat ideals are first introduced by Vickers  in the case that $\sQ$ is Lawvere's quantale $([0,\infty]^{\rm op},+)$ and are called \emph{flat left modules} \cite{SV2005}. Flat ideals provide a natural analogy of ideals (= directed lower sets of partially ordered sets) in the $\sQ$-valued context. We haste to remark that there exist different extensions  of the notion of directed lower set to the enriched context,  with flat ideal  being one of them. A comparative study of some of these extensions   can be found in \cite{LZZ2020}.

\begin{prop}{\rm(\cite{LZZ2020,TLZ2014})} If $f\colon X\lra Y$ preserves $\sQ$-order, then for each flat ideal $\phi$ of $X$, $f^\ra(\phi)$ is a flat ideal of $Y$. \end{prop}
For each $\sQ$-ordered set $X$, write \[\CF X\] for the subset of $\CP X$ consisting of flat ideals. It is obvious that for each $a\in X$, the weight $X(-,a)$ is a flat ideal of $X$, so  the Yoneda embedding $\sy_X\colon X\lra\CP X$ factors through $\CF X$. The assignment $X\mapsto\CF X$ defines a \emph{saturated class of weights}  on $\QOrd$, see \cite[Proposition 4.5]{LZZ2020}. In other words, it gives rise to a submonad \[\mathbb{F}=(\CF,\sfm,\sy)\] of the monad \(\mathbb{P}=(\CP,\sfm,\sy).\)

A $\sQ$-ordered set $X$ is \emph{$\CF$-cocomplete} if every flat ideal of $X$ has a supremum. It is clear that $X$ is $\CF$-cocomplete if, and only if, the map \[\sy_X\colon X\lra\CF X, \quad x\mapsto X(-,x)\] has a left adjoint.

Separated and $\CF$-cocomplete $\sQ$-ordered sets are precisely the algebras of the monad $\mathbb{F}=(\CF,\sfm,\sy)$ (see e.g. \cite[Section 3]{LZ2020}) and provide an analogy of {\bf dcpo}s (= directed complete partially ordered sets) \cite{Gierz2003,Goubault} in the $\sQ$-valued context. Such a $\sQ$-ordered set  will be called an $\CF$-{\bf dcpo}. Since there exist different  extensions of directed lower sets, hence different extensions of {\bf dcpo}s to the enriched context  (see e.g. \cite{LZZ2020}), instead of   the term ``$\sQ$-{\bf dcpo}'', we adopt  ``$\CF$-{\bf dcpo}'' with the prefix $\CF$ indicating that such a ``$\sQ$-valued {\bf dcpo}'' is defined with respect to the class $\CF$ of flat weights. This remark also applies to $\CF$-domains defined below.

\begin{prop} \label{specialization Q-order is DC}
The specialization $\sQ$-order of a sober $\sQ$-topological space is $\CF$-cocomplete.  \end{prop}
\begin{proof}Let $X$ be a sober $\sQ$-topological space. We   show that every flat ideal $\phi$ in $\Omega X$ has a supremum.   Since each open set of $X$ is a coweight of $\Omega X$, it is readily verified that \[p\colon \CO(X)\lra\sQ,\quad p(\lambda)=\phi\pitchfork\lambda
\] is a point of the $\sQ$-module $\CO(X)$. Thus, there is some $a$ of $X$ such that $p(\lam)=\lam(a)$ for all $\lam\in\CO(X)$. We claim that $a$ is a supremum of $\phi$. Indeed, for all $b\in X$, \begin{align*}\Omega X(a,b) &= \bw_{\lam\in\CO(X)}\lam(a)\ra\lam(b) \\ &= \bw_{\lam\in\CO(X)}\Big(\bv_{x\in X}\phi(x)\with\lambda(x)\ra\lam(b)\Big) \\ &= \bw_{x\in X}\Big( \phi(x)\ra \bw_{\lam\in\CO(X)} (\lambda(x)\ra\lam(b)) \Big) \\ &= \sub_X(\phi,\Omega X (-,b)), \end{align*} hence $a$ is a supremum of $\phi$ in $\Omega X$. \end{proof}
\begin{defn}\label{defn of F-domain}
A $\sQ$-ordered set $ X $ is  an $\CF$-domain if $X$ is separated and there exists a string of adjunctions \[\thda\dashv\sup\dashv\sy_X\colon X\lra\CF X.\]
 \end{defn}

 Said differently, a   $\sQ$-ordered set $ X $ is an $\CF$-domain if it is an $\CF$-{\bf dcpo} and  is continuous in the sense that for each $x\in X$,  there is a flat ideal $\thda x$ such that for all $\phi\in\CF X$, \[\sub_X(\thda x,\phi)=X(x,\sup\phi). \]

When flat ideals are taken as directed lower sets  in the $\sQ$-valued context, $\CF$-domains are  then an analogy of  domains (= continuous {\bf dcpo}s)   \cite{Gierz2003,Goubault}.

\begin{exmp} For each $\sQ$-ordered set $X$, the set $\CF X$ of flat ideas in $X$ with the  inclusion $\sQ$-order is an $\CF $-domain. This is an instance of \cite[Proposition 3.3]{LZ2020} or \cite[Theorem 4.1]{LZZ2020}. Actually,
for each $\sQ$-ordered set $X$  we have  a string of adjunctions \[\sy_X^\ra\dashv\sy_X^\leftarrow\dashv\sy_{\CF X}:\CF X\to\CF\CF X,\] where $\sy_X\colon X\lra\CF X$ is the unit of the monad $\mathbb{F}=(\CF,\sfm,\sy)$, hence $\CF X$ is an $\CF $-domain. In particular, the supremum of a flat ideal $\Phi$ of $\CF X$ is given by \[\sup\Phi =\sy_X^\leftarrow(\Phi)=\bv_{\phi\in\CF X} \Phi(\phi)\with\phi.\]
\end{exmp}

\begin{rem}[When $\sQ$ is a frame, III]\label{frame-valued flat ideal} Assume that $\sQ=(\sQ,\with)$ is a frame, i.e., $\with=\wedge$. It follows from  \cite[Proposition 3.11]{TLZ2014} that for each $\sQ$-ordered set $X$,   a weight $\phi$  is a flat ideal if, and only if, it is inhabited and \[\phi(x)\wedge\phi(y)=\bv_{z\in X}\phi(z) \wedge X(x,z)\wedge X(y, z)\] for all $x,y\in X$. So, in this case, a flat ideal of $X$ is precisely an \emph{\emph{ideal}} of $X$ in the sense of \cite[Definition 5.1]{LZ2007}; the fuzzy dcpos and fuzzy domains studied in \cite{LiuZhao,Yao10,Yao12,Yao16} are $\CF$-cocomplete $\sQ$-ordered sets and $\CF$-domains, respectively.
\end{rem}



Let $X$ be an $\CF$-cocomplete $\sQ$-ordered set. The \emph{way below $\sQ$-relation} (relative to flat ideals)   on $X$ refers to the $\sQ$-relation $w\colon X\times X\lra \sQ$ defined by
$$w (x,y)=\bigwedge_{\phi\in\mathcal{F}X}X(y,\sup\phi)\ra\phi(x).$$

We list below some basic properties of the way below $\sQ$-relation $w$, they are an instance of the properties of quantitative domains based on  a class of weights, see e.g. \cite{HW2011,HW2012,Was2009}. For all $x,y,z,u\in X$,
\begin{enumerate}[label=\rm(\roman*)] \setlength{\itemsep}{0pt}
	\item  $w(x,y)\leq X(x,y)$;
	\item  $w(y,z)\with X(x,y)\leq w(x,z)$, in particular, $w(-,z)$ is a weight of $X$;
	\item  $X(z,u)\with w(y,z)\leq w(y,u)$, in particular, $w(y,-)$ is a coweight of $X$.
\item If $X$ is an $\CF$-domain, then $w$ is   interpolative in the sense that \[w(x,y)=\bv_{z\in X}w(z,y)\with w(x,z).\]
\item If $X$ is a separated $\sQ$-ordered set of which every flat ideal has a supremum, then $X$ is a $\CF$-domain if, and only if, for all $x\in X$, the weight $w(-,x)$ is a flat ideal with supremum $x$. In this case, the left adjoint $\thda\colon X\lra\mathcal{F}X$ of $\sup\colon\mathcal{F}X\lra X$ is given by $\thda x=w(-,x)$.
\end{enumerate}

\begin{exmp}[When $\sQ$ is a frame, IV] \label{frame-valued V}
Assume that $\sQ=(\sQ,\with)$ is a frame, i.e., $\with=\wedge$. Then  $\sQ$ together with the canonical $\sQ$-order is an $\CF$-domain; that is, $(\sQ,\alpha_L)$ is an $\CF$-domain. By \cite[Example 3.14]{Zhang2007} and Example \ref{frame-valued flat ideal}, the assignment $x\mapsto x\vee(\id\ra0)$ defines a left adjoint of the map that sends each flat ideal of $(\sQ,\alpha_L)$ to its supremum. Thus, $(\sQ,\alpha_L)$ is an $\CF$-domain. The way below $\sQ$-relation on $(\sQ,\alpha_L)$ is given by \(w(y,x)= x\vee (y\ra0).\)

For a general quantale $\sQ$, $(\sQ,\alpha_L)$ may fail to be an $\CF$-domain, see Theorem \ref{[0,1] is domain} below. \end{exmp}


\begin{defn} (\cite{LZZ2020}) \label{def of Scott topology}
A fuzzy set $ \psi\colon X\lra\sQ $  of a  $\sQ$-ordered set $X$ is    Scott open if it is a coweight and     \[\psi(\sup\phi)\leq\phi\pitchfork\psi\]  for each flat ideal $ \phi $ of $X$  whenever $\sup\phi$ exists.
\end{defn}

Since for any weight $\phi$ and any coweight $\psi$ of $X$,  \[k\leq X(\sup\phi,\sup\phi)=\CP X(\phi, X(-,\sup\phi))\] whenever $\sup\phi$ exists, it  follows that \[\phi(x) \leq X(x,\sup\phi)\leq \psi(x)\ra\psi(\sup\phi)\]  for all $x\in X$, hence $\phi\pitchfork\psi\leq \psi(\sup\phi)$. Therefore, the inequality in Definition \ref{def of Scott topology} is actually an equality.


Given a $\sQ$-ordered set $ X $,   the family of  Scott open fuzzy sets form a $\sQ$-topology  on $X$ \cite[Proposition 5.2]{LZZ2020}, called the   Scott $\sQ$-topology on $ X $ and  denoted by $ \sigma(X) $. We write \[\Sigma X\] for the $\sQ$-topological space $(X,\sigma(X))$.

Now we state the main result of this section.
\begin{thm}\label{domain is sober} Let $\sQ$ be an integral and commutative quantale. Then the Scott $\sQ$-topology of each $\CF$-domain  is sober.
\end{thm}

\begin{lem}\label{basis for the Scott} Let $X$ be an $\CF$-domain and let $\psi\colon X\lra(\sQ,\alpha_L)$ be a $\sQ$-order-preserving map, viewed as a fuzzy subset of $X$. Then the interior of $\psi$ in $\Sigma X$ is given by \[\psi^\circ=\bv_{a\in X}  w(a,-)\with\psi(a).\] 
\end{lem}

\begin{proof}  First,  making use of the interpolation property of the way below $\sQ$-relation, one sees that for each $x\in X$,  the fuzzy set \(w(x,-)\colon X\lra\sQ\) is  Scott open.

Next, we show that $\psi$ is open if, and only if, \[\psi=\bv_{a\in X}  w(a,-)\with\psi(a).\] The   conclusion  follows from this fact immediately. Sufficiency is clear since each $w(a,-)$ is open. To see the necessity, we check that  for all $x\in X$, \[\psi(x)=\bv_{a\in X} w(a,x)\with\psi(a).\]

Since $X$ is an $\CF$-domain, $w(-,x)$ is a flat ideal with $x$ as a supremum. Since $\psi$ is Scott open, then \begin{align*}\psi(x) &= \psi(\sup w(-,x))  =  w(-,x)\pitchfork \psi     =  \bv_{a\in X} w(a,x)\with\psi(a).\qedhere \end{align*}
\end{proof}

\begin{proof}[Proof of Theorem \ref{domain is sober}] Write $\mathcal{O}(X)$ for the set of open sets of the  $\sQ$-topological space $\Sigma X$, viewed as a $\sQ$-module. Let $p\colon \mathcal{O}(X)\lra  \sQ$ be a point of $\mathcal{O}(X)$. We need to show that there is a unique element $b$ of $X$ such that for each open set $\psi$ of $\Sigma X$, $p(\psi)=\psi(b)$.

Define $\phi\colon X\lra\sQ$ by \[\phi(a)=p(w(a,-)).\] We claim that $\phi$ is a flat ideal and the supremum of $\phi$ satisfies the requirement.

\textbf{Step 1}. $\phi$ is an  inhabited   weight of $X$. Since $\sQ$ is integral, the top element $1$ of $\sQ$ is the unit of $\&$, then $$1_X=1_X^\circ =\bv_{a\in X} w(a,-)\with1=\bv_{a\in X} w(a,-),$$  thus $$\bv_{a\in X}\phi(a) = \bv_{a\in X}p((w(a,-))=p\Big(\bv_{a\in X}w(a,-)\Big)=p(1_X)=1,$$ which shows that $\phi$ is inhabited. That $\phi$ is a weight follows from that for all $a,b\in X$, \[\phi(b)\with X(a,b)=p(w(b,-))\with X(a,b)= p(w(b,-)\with X(a,b))\leq p(w(a,-))=\phi(a). \]

\textbf{Step 2}. $\phi$ is flat. Since for each coweight $\psi$ of $X$, \begin{align*}\phi\pitchfork\psi &= \bv_{a\in X} p(w(a,-))\with \psi(a) \\ &=\bv_{a\in X} p( w(a,-)\with \psi(a)) \\ &= p\Big(\bv_{a\in X}  w(a,-)\with \psi(a)\Big)\\ &= p(\psi^\circ), \end{align*} it follows that for any coweights $\psi_1,\psi_2$   of $X$,  \begin{align*} \phi\pitchfork(\psi_1\wedge\psi_2) &= p((\psi_1\wedge\psi_2)^\circ) \\ &= p(\psi_1^\circ)\wedge p(\psi_2^\circ) \\ &= (\phi\pitchfork\psi_1)\wedge(\phi\pitchfork\psi_2), \end{align*} hence $\phi$ is flat.

\textbf{Step 3}. Let $b$ be the supremum of the flat ideal $\phi$. Then for each open set $\psi$ of $\Sigma X$, \begin{align*} p(\psi) &= p\Big(\bv_{a\in X}\  w(a,-)\with \psi(a)\Big)\\ & =\bv_{a\in X}  p(w(a,-)\with \psi(a))\\ & = \bv_{a\in X} \phi(a)\with \psi(a) \\ &=\phi\pitchfork\psi\\ & =\psi(\sup\phi) \\ &=\psi(b).\end{align*} The proof is completed. \end{proof}

\begin{exmp}[When $\sQ$ is a frame, V] \label{frame-valued VI} (\cite[Theorem 3.5]{Yao16})
Assume that $\sQ=(\sQ,\with)$ is a frame, i.e., $\with=\wedge$. By Example \ref{frame-valued V}, $(\sQ,\alpha_L)$ is an $\CF$-domain  and $w(x,y)= y\vee (x\ra0)$ for all $x,y\in\sQ$. Then it follows from Lemma \ref{basis for the Scott} that  $\Sigma(\sQ,\alpha_L)$ coincides with the Sierpi\'{n}ski space $\mathbb{S}$, hence $\mathbb{S}=\Sigma\Omega\thinspace\mathbb{S}$ in this case.   But, this is not true   for a general quantale, see Proposition \ref{Sierpinski is Scott} in the next section.
\end{exmp}

\begin{rem}In the case that  the quantale $\sQ$ is a frame, Theorem \ref{domain is sober} and Proposition \ref{specialization Q-order is DC} are first proved  in \cite{Yao12}. \end{rem}


\section{Some  examples in the case that $\sQ=([0,1],\with)$}
The aim of this section is to present some examples of $\CF$-domains and sober $\sQ$-topological spaces  when $\sQ$ is the interval $[0,1]$ endowed with a continuous t-norm. In particular, a necessary and sufficient condition  is obtained for $[0,1]$ together with the canonical $\sQ$-order to be an $\CF$-domain.

A \emph{left-continuous t-norm} on $[0,1]$ \cite{Klement2000} is a  map $\with\colon [0,1]^2\lra[0,1]$  that makes $([0,1],\with)$ into a  commutative and integral quantale.
A  left-continuous t-norm that is continuous with respect to the usual topology is called a \emph{continuous t-norm}  \cite{Klement2000}. Continuous t-norms play a decisive role in the BL-logic of H\'{a}jek \cite{Ha98}.

Basic   continuous t-norms and their implication operators are listed below:
\begin{enumerate}[label={\rm(\arabic*)}] \setlength{\itemsep}{0pt}
\item  The G\"{o}del t-norm:  \[ x\with y= \min\{x,y\}; \quad x\ra y=\begin{cases}
		1,&x\leq y,\\
		y,&x>y.
		\end{cases} \] The implication  operator $\ra$ of the G\"{o}del t-norm is  continuous  except at   $(x,x)$, $x<1$.

\item  The product t-norm:   \[ x\with_P y=xy; \quad x\ra y=\begin{cases}
		1,&x\leq y,\\
		y/x,&x>y.
		\end{cases} \] The implication operator  $\ra$
of the product t-norm is  continuous  except at   $(0,0)$.

\item  The {\L}ukasiewicz t-norm: \[ x\with_{\text \L} y=\max\{0,x+y-1\}; \quad x\ra y=\min\{1-x+y,1\}. \] The implication operator  $\ra$
    of the {\L}ukasiewicz t-norm
    is   continuous on $[0,1]^2$.
	\end{enumerate}

Let $\with $ be a continuous t-norm. An element $p\in [0,1]$ is   \emph{idempotent}  if $p\with p=p$.

\begin{prop}{\rm(\cite[Proposition 2.3]{Klement2000})} \label{idempotent}
Let $\&$ be a continuous t-norm on $[0,1]$ and $p$ be an idempotent element of $\&$. Then $x\with y=  x\wedge y $ whenever $x\leq p\leq y$.
 \end{prop} It follows immediately  that $y\ra x=x$ whenever  $x< p\leq y$ for some idempotent   $p$. Another consequence of  Proposition \ref{idempotent} is that for any idempotent elements $p, q$    with $p<q$,  the restriction of $\with $ to $[p,q]$, which is also denoted by $\with$,  makes $[p,q]$ into a commutative quantale with $q$ being the unit element.

A continuous t-norm on $[0,1]$ is said to be \emph{Archimedean} if it has no idempotent element  other than $0$ and $1$. It is well-known  that if $\with$ is a continuous Archimedean t-norm, then the quantale  $([0,1],\with)$ is either isomorphic to  $([0,1],\with_{\text \L})$ or to $([0,1],\with_{P})$, see e.g.  \cite{Klement2000,Mostert1957}.

\begin{thm} {\rm(\cite{Klement2000,Mostert1957})}
\label{ordinal sum} Let $\with $ be a continuous t-norm. Then for each non-idempotent element $c$ of $([0,1],\with)$, there exist idempotent elements $c^{-}$ and $c^{+}$ such that $c^-<c<c^+$ and  the quantale  $([c^{-},c^{+}],\with)$ is either isomorphic to  $([0,1],\with_{\L})$ or to $([0,1],\with_{P})$.
\end{thm}

\begin{prop}\label{[0,1] is domain} Let $\with $ be a continuous t-norm and let $\sQ=([0,1],\with)$. Then, the following are equivalent:
\begin{enumerate}[label=\rm(\arabic*)] \setlength{\itemsep}{0pt}
	\item  The $\sQ$-ordered set $([0,1],\alpha_L)$ is an $\CF$-domain.
	\item  For each non-idempotent element $c\in[0,1]$, the quantale $([c^-,c^+],\with)$ is isomorphic to $([0,1],\with_P)$ whenever $c^->0$.
\end{enumerate}  \end{prop}

In order to prove the conclusion, we present four lemmas first.
\begin{lem} {\label{fidealinA}}
Let $ \with  $ be a continuous Archimedean t-norm and let $\sQ=([0,1],\with)$. Then $\phi\colon[0,1]\lra[0,1] $ is a flat ideal of the $\sQ$-ordered set $ ([0,1],\alpha_L) $ if, and only if, $ \phi=\alpha_L(-,a)$ for some $ a\in[0,1] $.
\end{lem}

\begin{proof}A direct consequence of Corollary 3.14 and Theorem 3.18 in \cite{LZZ2020}. \end{proof}

\begin{lem}\label{f1} Let $\with $ be a continuous t-norm and let $\sQ=([0,1],\with)$.	Let $ \phi $ be a flat ideal of  $([0,1],\alpha_L)$. If $c\in[0,1]$ is   idempotent and $c\leq\phi(c)$, then   $\phi(c) $ is idempotent.
\end{lem}

\begin{proof} Consider the following coweights of $ ([0,1],\alpha_L)$: $\psi_1\equiv\phi(c) $ and $\psi_2  =\alpha_L(c,-)$.

Since $ \phi\pitchfork\psi_1=\phi(c) =\phi\pitchfork\psi_2$ and   \[ (\psi_1\wedge\psi_2)(x)=\begin{cases} 	c\ra x,&0\leq x<c,\\ 	\phi(c),&c\leq x\leq 1, 	\end{cases} \] it follows that \begin{align*} \phi(c)&	= \phi\pitchfork(\psi_1\wedge\psi_2)\\ &= \Big(\bv_{x\in[0,c)}\phi(x)\with (c\ra x)\Big)\vee\Big(\bv_{x\in[c,1]}\phi(x)\with \phi(c)\Big)\\ &= \Big(\bv_{x\in[0,c)}(c\ra x)\Big)\vee(\phi(c)\with \phi(c)) \quad\quad (\forall x\in[0,c), c\ra x\leq c\leq \phi(c)\leq\phi(x))\\ 	 	&\leq c\vee(\phi(c)\with \phi(c))\\ 	&=\phi(c)\with \phi(c),   \quad\quad (\text{$c\leq \phi(c)$}) 	\end{align*} hence $\phi(c)$ is idempotent.
\end{proof}

For a non-idempotent element $c$ in $([0,1],\with)$, write $\sQ^c$ for the quantale obtained by restricting $\with$ to $[c^-,c^+]$; that is, $$\sQ^c=([c^-,c^+],\with).$$

\begin{lem}\label{f2} Let $\with $ be a continuous t-norm, let $\sQ=([0,1],\with)$, and let $ \phi $ be a flat ideal of   $([0,1],\alpha_L)$. If $c\in[0,1]$ is  non-idempotent and $\sup\phi\geq c$, then  the map \[\rho\colon [c^-,c^+]\lra[c^-,c^+], \quad \rho(t)= c^+\wedge\phi(t)\] is a flat ideal of the $\sQ^c$-ordered set $([c^-,c^+],\alpha_L^c)$, where $\alpha_L^c$ refers to the canonical $\sQ^c$-order on $[c^-,c^+]$.
\end{lem}
\begin{proof}
Since $\phi(c^+)\geq\phi(1)=\sup\phi=c >c^-$, it follows that $\phi(x)\geq c^-$ for all $x\in[c^-,c^+]$, hence $\rho$ is well-defined. That $\rho$ is a weight of $([c^-,c^+],\alpha_L^c)$ is clear,  it remains to check that it is inhabited and flat.

\textbf{Step 1}. Since $\phi(c^-)>c^-$, it follows from Lemma \ref{f1} that $\phi(c^-)$ is idempotent, hence $\phi(c^-)\geq c^+$ and  $\rho(c^-)=c^+$, which shows that $\rho$ is inhabited.

\textbf{Step 2}.	For each coweight $\psi$ of $([c^-,c^+],\alpha_L^c)$, the map $ \widetilde{\psi}\colon [0,1]\lra[0,1] $, given by
	\[ \widetilde{\psi}(x)=\begin{cases}
	x,&x\in[0,c^-),\\
	\psi(x),&x\in[c^-,c^+],\\
	\psi(c^+),&x\in(c^+,1],
	\end{cases} \] is a coweight of $ ([0,1],\alpha_L)$.
The verification is straightforward and  omitted here.

\textbf{Step 3}. $\rho$ is flat; that is, for any  coweights $\psi_1,\psi_2$ of   $([c^-,c^+],\alpha_L^c)$, \[\rho\pitchfork(\psi_1\wedge\psi_2) =(\rho\pitchfork\psi_1)\wedge(\rho\pitchfork\psi_2).\]
	
First, a routine calculation shows that $\rho\pitchfork\psi=\phi\pitchfork\widetilde{\psi}$ for all coweight $\psi$ of $([c^-,c^+],\alpha_L^c)$.
 Thus,	\begin{align*}	\rho\pitchfork(\psi_1\wedge\psi_2) &=\phi\pitchfork(\widetilde{\psi_1\wedge\psi_2})\\ &=\phi\pitchfork(\widetilde{\psi_1}\wedge\widetilde{\psi_2})\\ 	&=(\phi\pitchfork\widetilde{\psi_1})\wedge(\phi\pitchfork\widetilde{\psi_2}) &(\phi\text{ is flat})\\ 	&=(\rho\pitchfork \psi_1)\wedge(\rho\pitchfork\psi_2),
	\end{align*}
which completes the proof.
\end{proof}

\begin{lem}\label{smallest idea}
Let $\with $ be a continuous t-norm and let $\sQ=([0,1],\with)$. For each $ x\in[0,1] $,  define $d(x)\colon[0,1]\lra[0,1]$  by \[   d(x)(t)=\begin{cases}t\ra0,&x=0,\\
	1,&x>0,\thinspace t=0,\\
	x^+\wedge(t\ra x),&x>0,\thinspace t>0.
	\end{cases} \] Then $d(x)$  is the smallest flat ideal of $ ([0,1],\alpha_L) $ with supremum larger than or equal to $x$.
\end{lem}

\begin{proof} The verification of that $d(x)$ is a flat ideal with supremum $x$ is straightforward. Now we check that for any flat ideal $\phi$ of  $([0,1],\alpha_L)$ with $\sup\phi\geq x$, it holds that $d(x)\leq \phi $.
The conclusion is obvious if $x$ is idempotent. Assume that $x$ is not idempotent.   By Lemma \ref{f2}, the map \[\rho\colon [x^-,x^+]\lra[x^-,x^+], \quad \rho(t)= x^+\wedge\phi(t)\] is a flat ideal of the $\sQ^x$-ordered set $([x^-,x^+],\alpha_L^x)$. 	Since  the quantale $([x^-,x^+],\with)$ is   isomorphic to $[0,1]$ endowed with a continuous Archimedean t-norm, it follows from Lemma \ref{fidealinA} that there is some $b\in[x^-,x^+]$ such that  \[\rho(t)=\alpha_L^x(t,b)=x^+\wedge(t\ra b)\] for all $t\in[x^-,x^+]$.

We claim that $b\geq x$, hence $d(x)(t)\leq \rho(t)\leq\phi(t)$ for all $t\in[x^-,x^+]$. 	Suppose on the contrary that $b<x$. Take some $z\in(x,x^+)$ such that $z\ra b<x$. Then $x> \rho(z)=x^+\wedge\phi(z)$, hence $\phi(z)=\rho(z)<x$, contradicting that $\phi(z)\geq\phi(1)=\sup\phi\geq x$ (the equality holds by Example \ref{sup in Q}).

Finally, since   \[d(x)(t)= x^+= \rho(x^-)\leq\phi(t)\] for all $t\in(0,x^-)$  and \[d(x)(t)= x\leq\phi(1)\leq\phi(t)\] for all $t\in(x^+,1]$, it follows that $d(x)\leq\phi$, as desired.
\end{proof}

\begin{proof}[Proof of Proposition \ref{[0,1] is domain}] Write $X$ for the $\sQ$-ordered set $([0,1],\alpha_L)$ and for each $x\in X$, let $d(x)$ be the smallest flat ideal of $X$ with supremum   larger than or equal to $x$, as given in Lemma \ref{smallest idea}. Then,
by Theorem \ref{Characterization of adjoints}, it is sufficient to show that the map \(d\colon X\lra \mathcal{F}X \)
preserves $\sQ$-order if, and only if,
$\with $ satisfies the condition stated in (2). We check the only-if-part here and leave the  if-part  to the reader.

Suppose on the contrary that  $\with $ does not satisfy that condition. Then there exist idempotent elements $p,q>0$  such that the restriction of $\with$ on $[p,q]$ is isomorphic to the \L ukasiewicz t-norm. Pick $t\in(p,q)$. Then    \begin{align*}X(t,p)& = t\ra p\\ &>q\ra p\\ &\geq \bw_{x\in(0,1]}\big(q\wedge(x\ra t)\ra p\wedge(x\ra p)\big)\\  &=\mathcal{F}X(d(t), d(p)), \end{align*} which contradicts that \(d\colon X\lra \mathcal{F}X \)  preserves  $\sQ$-order. \end{proof}

\begin{prop}Let $\with$ be a continuous t-norm on $[0,1]$  and let $\sQ=([0,1],\with)$. If for each non-idempotent element $c\in[0,1]$, the quantale $([c^-,c^+],\with)$ is isomorphic to $([0,1],\with_P)$ whenever $c^->0$, then the $\sQ$-topological space $\Sigma([0,1],\alpha_L)$ is sober.\end{prop}

\begin{proof}This follows from a combination of Proposition \ref{[0,1] is domain} and Theorem \ref{domain is sober}. \end{proof}

A celebrated result of Scott \cite{Scott72} says that if $X$ is an injective $T_0$ topological space, then the topology of $X$ is precisely the Scott topology of its specialization order. It is shown in \cite{Yao16} that this is also true in the $\sQ$-enriched setting if the quantale $\sQ$ is a frame. But, this is not true in general, as we   see below.

\begin{prop}\label{Sierpinski is Scott} Let $\with $ be a continuous t-norm and let $\sQ=([0,1],\with)$. Let $\mathbb{S}$ be the Sierpi\'{n}ski $\sQ$-topological space. Then $\mathbb{S}=\Sigma\Omega\thinspace\mathbb{S}$ if, and only if, $\with$ is the G\"{o}del t-norm. \end{prop}
A lemma first.
\begin{lem}\label{Scott open set of dL} Let $\with $ be a continuous t-norm and let $\sQ=([0,1],\with)$.  Then a coweight $\psi$ of $([0,1],d_L)$ is Scott open   if, and only if, for all $x\in (0,1]$,   $ \psi(x)>x^+ $ implies $ \psi(x)=\psi(0)$.	
\end{lem}

\begin{proof} This is the content of \cite[Lemma 4.11]{ZZ2019}. We present here a  simple proof   with help of Lemma \ref{smallest idea}. For each $x\in [0,1]$, let $d(x)$ be the smallest flat ideal of $X$ with supremum   larger than or equal to $x$. For each coweight $\psi$ of $([0,1],d_L)$, since we always have \[\psi(0)=d(0)(0)\with\psi(0)\leq d(0)\pitchfork\psi,\] it follows that   $\psi$  is Scott open   if, and only if, for all $x\in(0,1]$, $\psi(x)\leq d(x)\pitchfork\psi$. Because  \begin{align*}d(x)\pitchfork\psi &= \psi(0)\vee\bv_{t>0}(x^+\wedge(t\ra x))\with\psi(t)\\ & = \psi(0)\vee\bv_{t>0} x^+\with(t\ra x) \with\psi(t) &\text{($x^+$ is idempotent, Proposition \ref{idempotent})}\\ &=\psi(0)\vee(x^+\wedge\psi(x)),\end{align*} therefore, $\psi$ is Scott open if,  and only if,  for each $x\in (0,1] $,   $ \psi(x)>x^+$ implies  $\psi(x)=\psi(0)$. \end{proof}

\begin{proof}[Proof of Proposition \ref{Sierpinski is Scott}] Sufficiency follows from Example \ref{frame-valued VI}. To see the necessity, suppose on the contrary that $([0,1],\with)$ has a non-idempotent element, say $c$.
Then the map \[\lam(x)\coloneqq ((c\ra c^-)\vee(c\ra x))\wedge c^+\]
satisfies the conditions in Lemma \ref{Scott open set of dL}, hence it is a Scott open set of $\Omega\thinspace\mathbb{S}$. But, $\lam$ is not an open set of $\mathbb{S}$.
\end{proof}
Since for each commutative and unital quantale $\sQ$, the space $\mathbb{S}$ is an injective $\sQ$-topological space, Proposition \ref{Sierpinski is Scott}  shows that in the quantale-enriched context, the $\sQ$-topology of an injective $\sQ$-topological space may not be the Scott $\sQ$-topology of its specialization $\sQ$-order.


We end this section with a discussion of the sobriety of the space $\Sigma([0,1],\alpha_R)$, where $\alpha_R$ is the opposite of the canonical $\sQ$-order on $[0,1]$.

\begin{lem}{\rm(\cite[Lemma 4.14]{ZZ2019})} \label{fl}
Let $\with $ be a continuous t-norm and let $\sQ=([0,1],\with)$. Then, for each coweight $\psi$ of $([0,1],\alpha_R)$, it holds that
\begin{enumerate}[label={\rm(\roman*)}] \setlength{\itemsep}{0pt}
\item if $ \psi(x)\leq x^- $ then $ \psi(x)=\psi(1) $;
\item if $x$ is idempotent and $\psi(x)\geq x$ then $\psi(1)\geq x$.
\end{enumerate}
\end{lem}

\begin{lem} {\rm(\cite[Lemma 4.15]{ZZ2019})} \label{openindR} Let $\with $ be a continuous t-norm and let $\sQ=([0,1],\with)$. If $ b $ is a nontrivial idempotent element in $ ([0,1],\&) $, then for each  Scott open  set $\psi$ of $([0,1],\alpha_R)$, either $\psi$ is a constant map or $\psi(x)\leq b $ for all $ x\in[0,1] $. \end{lem}

\begin{prop}\label{dR is a domain} Let $\with $ be a continuous t-norm and let $\sQ=([0,1],\with)$. Then the following  are equivalent:
\begin{enumerate}[label=\rm(\arabic*)] \setlength{\itemsep}{0pt}
	\item  The $\sQ$-ordered set $([0,1],\alpha_R)$ is an $\CF$-domain.
\item The $\sQ$-topological space  $\Sigma([0,1],\alpha_R)$ is sober.
	\item  $\with$ is Archimedean.
\end{enumerate}  \end{prop}

\begin{proof} Write $X$ for the $\sQ$-ordered set $([0,1],\alpha_R)$.

$(1)\Rightarrow(2)$ Theorem \ref{domain is sober}.

$(2)\Rightarrow(3)$ Suppose on the contrary that  $ ([0,1],\&) $ has a nontrivial idempotent element, say $b$. By Lemma \ref{fl} and Lemma \ref{openindR}, for all open set $\psi$ of $\Sigma X$ and   all $x,y>b$, it holds that $\psi(x)=\psi(y)$. Therefore, $\Sigma X$ is not a $T_0$ space, hence not sober,   contradicting the assumption.

$(3)\Rightarrow(1)$ By Corollary 3.14 and Theorem 3.18 in \cite{LZZ2020}, a weight $\phi$ of $X$ is a flat ideal if, and only if, $\phi =X(-,a)$ for some $a\in[0,1]$. By help of this fact, one readily verifies that \(a\mapsto X(-,a)\) defines a left adjoint of $\sup\colon\CF X\lra X$.
 \end{proof}



\begin{thebibliography}{[99]}\setlength{\itemsep}{0pt}



\bibitem{Borceux1994} F. Borceux,  {\em Handbook of Categorical Algebra, Vol. 2},  Cambridge University Press, Cambridge, 1994.
\bibitem{EGHK}P. Eklund, J. Guti\'{e}rrez Garc\'{i}a, U. H\"{o}hle, J. Kortelainen, \emph{Semigroups in
Complete Lattices. Quantales, Modules and Related Topics}, Springer, 2018.
\bibitem{Gierz2003} G. Gierz, K.H. Hofmann, K. Keimel, J.D. Lawson, M. Mislove,   D.S. Scott, \emph{Continuous Lattices and Domains},   
    Cambridge University Press, Cambridge, 2003.


\bibitem{Goubault} J. Goubault-Larrecq,   \emph{Non-Hausdorff Topology and Domain Theory}, Cambridge University Press, Cambridge, 2013.

\bibitem{GHP}J. Guti\'{e}rrez Garc\'{i}a, U. H\"{o}hle, M.A. de Prada Vicente, On lattice-valued frames: The completely distributive case, Fuzzy Sets and Systems 161 (2010) 1022-1030.
\bibitem{Ha98} P. H\'{a}jek,  \emph{Metamathematics of Fuzzy Logic}, Kluwer Academic Publishers, Dordrecht, 1998.


\bibitem{HW2011}D. Hofmann, P. Waszkiewicz, Approximation in quantale-enriched categories, Topology and its Applications 158 (2011)  963-977.

\bibitem{HW2012} D. Hofmann, P. Waszkiewicz,    A duality of quantale-enriched categories, Journal of Pure and Applied Algebra 216 (2012) 1866-1878.


\bibitem{JY2016}G. J\"{a}ger, W. Yao, Completely prime $L$-filters, irreducible $L$-filters and sobriety, Quaestiones Mathematicae  39 (2016) 831-844.

\bibitem{Johnstone} P.T. Johnstone, Stone Spaces, Cambridge University Press, Cambridge, 1982.
\bibitem{Joyal-Tierney}A. Joyal,  M. Tierney,   An Extension of the Galois Theory of Grothendieck, Memoirs of the American Mathematical Society, Vol. 51, No. 309. American Mathematical Society, Providence, Rhode Island, 1984.
\bibitem{KS2005} G.M. Kelly, V. Schmitt, Notes on enriched categories with colimits of some class, Theory and Applications of Categories 14 (2005) 399-423.
\bibitem{Klement2000} E.P. Klement, R. Mesiar, E. Pap,  \emph{Triangular Norms}, 
    Kluwer Academic Publishers, Dordrecht, 2000.

\bibitem{Kotze97}W. Kotz\'{e}, Fuzzy sobriety and fuzzy Hausdorff, Quaestiones Mathematicae 20 (1997) 415-422.


\bibitem{LZ2006} H. Lai, D. Zhang, Fuzzy preorder and fuzzy topology, Fuzzy Sets and Systems 157 (2006) 1865-1885.

\bibitem{LZ2007} H. Lai, D. Zhang, Complete and directed complete $\Omega$-categories, Theoretical Computer Science 388 (2007) 1-25.

\bibitem{LZ2009}	H. Lai, D. Zhang, Concept lattices of fuzzy contexts:  formal concept analysis vs. rough set theory, International Journal of Approximate Reasoning 50 (2009) 695-707.

\bibitem{LZ2020} H. Lai, D. Zhang, Completely distributive enriched categories are not always continuous, Theory and Application of Categories 35 (2020) 64-88.

\bibitem{LZZ2020} H. Lai, D. Zhang, G. Zhang, A comparative study of ideals in fuzzy orders, Fuzzy Sets and Systems 382 (2020) 1-28.

\bibitem{Lawvere1973}  F.W. Lawvere, Metric spaces, generalized logic, and closed categories, Rendiconti del Seminario Mat\'{e}matico e Fisico di Milano 43 (1973) 135-166.

\bibitem{LiuZhao}M. Liu, B. Zhao, Two Cartesian closed subcategories of fuzzy domains, Fuzzy Sets and Systems  238 (2014) 102-112.

\bibitem{MacLane1998} S. Mac\thinspace Lane, \emph{Categories for the Working Mathematician},   Graduate Texts in Mathematics, Vol. 5, Springer, New York, second edition, 1998.
\bibitem{Mostert1957} P.S. Mostert, A.L. Shields, On the structure of semigroups on a compact manifold with boundary, Annals of Mathematics  65 (1957) 117-143.	



\bibitem{PR08a}A. Pultr, S.E. Rodabaugh, Categorical theoretic aspects of chain-valued frames,  Part I:  Categorical and presheaf theoretic foundations, Fuzzy Sets and Systems 159 (2008) 501-528.

\bibitem{PR08b} A. Pultr, S.E. Rodabaugh, Categorical theoretic aspects of chain-valued frames,  Part II:  Applications to lattice-valued topology, Fuzzy Sets and Systems 159 (2008) 529-558.
\bibitem{Rodabaugh} S.E. Rodabaugh, Point-set lattice-theoretic topology, Fuzzy Sets and Systems 40 (1991) 296-345.
\bibitem{Rosenthal1990}   K.I. Rosenthal, \emph{Quantales and Their Applications}, Longman, Essex, 1990.

\bibitem{Scott72} D.S. Scott,   Continuous lattices, in:  
    \emph{Toposes, Algebraic Geometry and Logic}, Lecture Notes in Mathematics, vol. 274, pp. 97-136. Springer, Berlin, 1972.
\bibitem{SS16}S.K. Singh, A.K. Srivastava, On $Q$-sobriety, Quaestiones Mathematicae 39 (2016) 179-188.


\bibitem{SK98}A.K. Srivastava, A.S. Khastgir, On fuzzy sobriety, Information Sciences 110 (1998) 195-205.
\bibitem{Stubbe2005} I. Stubbe, Categorical structures enriched in a quantaloid:  categories, distributors and functors, Theory and Applications of Categories 14 (2005) 1-45.

\bibitem{Stubbe2006} I. Stubbe,   Categorical structures enriched in a quantaloid:  tensored and   cotensored categories, Theory and Applications of Categories 16 (2006) 283-306.

\bibitem{TLZ2014} Y. Tao, H. Lai,  D. Zhang, Quantale-valued preorders:  Globalization and cocompleteness, Fuzzy Sets and Systems  256 (2014) 236-251.
\bibitem{SV2005} S. Vickers, Localic completion of generalized metric spaces, Theory and Application of Categories 14 (2005) 328-356.

\bibitem{Was2009}P. Waszkiewicz, On domain theory over Girard quantales, Fundamenta Informaticae 92 (2009) 1-24.

\bibitem{Yao10} W. Yao, Quantitative domains via fuzzy sets: Part I: Continuity of fuzzy directed complete posets, Fuzzy Sets and Systems 161 (2010) 973-987.

\bibitem{Yao11} W. Yao, An approach to fuzzy frames via fuzzy posets, Fuzzy Sets and Systems 166 (2011) 75-89.
\bibitem{Yao12} W. Yao, A survey of fuzzifications of frames, the Papert-Papert-Isbell adjunction and sobriety, Fuzzy Sets and Systems 190 (2012) 63-81.
\bibitem{Yao16} W. Yao, A categorical isomorphism between injective fuzzy $T_0$-spaces and fuzzy continuous lattices, IEEE Transactions on Fuzzy Systems  24 (2016) 131-139.



\bibitem{Zhang2007} D. Zhang, An enriched category approach to many valued topology, Fuzzy Sets and Systems 158 (2007) 349-366.
\bibitem{Zhang2018} D. Zhang, Sobriety of quantale-valued cotopological spaces, Fuzzy Sets and Systems 350 (2018) 1-19.	


\bibitem{ZL95}D. Zhang, Y. Liu, $L$-fuzzy version of Stone's representation theorem for distributive lattices, Fuzzy Sets and Systems 76 (1995) 259-270.
\bibitem{ZZ2019}	D. Zhang, G. Zhang, Continuous triangular norm based fuzzy topology, Archive for Mathematical Logic 58 (2019) 915-942.
\end{thebibliography}
\end{document}